\documentclass[12pt]{elsarticle}
\usepackage{amsfonts}
\usepackage{amsmath}
\usepackage{amssymb}
\usepackage{graphicx}
\usepackage[tableposition=top]{caption}
\usepackage{subcaption}
\usepackage{float}
\usepackage{setspace}
\usepackage{placeins}

\providecommand{\U}[1]{\protect \rule{.1in}{.1in}}
\newtheorem{theorem}{Theorem}

\newtheorem{definition}{Definition}

\newtheorem{lemma}{Lemma}

\newtheorem{proposition}{Proposition}
\newtheorem{remark}{Remark}

\numberwithin{equation}{section}
\newenvironment{proof}{\paragraph{Proof:}}{\hfill$\square$}
\journal{...}

\input{tcilatex}

\begin{document}

\begin{frontmatter}
\title{The global existence of solutions and their asymptotic stability for a reaction-diffusion system}
\author[Teb]{Salem Abdelmalek}
\author[Tai2]{Samir Bendoukha}
\author[Roc]{Mokhtar Kirane}
\address[Teb]{Department of Mathematics, University of Tebessa 12002 Algeria. sallllm@gmail.com}
\address[Tai2]{Department of Electrical Engineering, College of Engineering,
Yanbu, Taibah University, Saudi Arabia. sbendoukha@taibahu.edu.sa}
\address[Roc]{LaSIE, Facult\'{e} des Sciences, Pole Sciences et Technologies, Universit\'{e} de La Rochelle, Avenue M. Cr\'{e}peau, 17042 La Rochelle Cedex, France. mokhtar.kirane@univ-lr.fr.}
\begin{abstract}
This paper studies the solutions of a reaction--diffusion system with
nonlinearities that generalize the Lengyel--Epstein and FitzHugh--Nagumo
nonlinearities. Sufficient conditions are derived for the global asymptotic
stability of the system's solutions. Furthermore, we present some numerical
examples.
\end{abstract}
\end{frontmatter}

\section{Introduction}

Reaction--diffusion systems are of great importance in many scientific and
engineering fields due to their ability to model numerous real life
phenomena. one of the most interesting of these phenomena is that of
morphogenesis, which is the biological process that causes organisms to
develop specific shapes and patterns. One of the very early works on
morphogenesis was conducted by Alan Turing in 1952 \citep{Turing1952}, where
he anticipated that diffusion driven instability in reaction--diffusion
systems leads to pattern formation. This theory was confirmed by the
concrete experiment of De--Kepper \textit{et al.} \citep{DeKepper1990} many
decades later through the chlorite--iodide--malonic acid--starch (CIMA)
chemical reaction in an open unstirred gel reactor. A model referred to as
Lengyel--Epstein was soon developped for the experiment in %
\citep{Lengyel1991,Lengyel1992}.

Many studies have been carried out on the dynamics of the Lengyel--Epstein
model. The study of Jang \textit{et al.} \citep{Jang2004} concerns the
global bifurcation structure of the set of non-constant steady states in the
one-dimensional case. The study in \citep{Ni2005} looked further into the
analytic characteristics of the model and showed that the initial
concentrations, size of reactor, and diffusion rates must be sufficiently
large to achieve Turing instability. The precise conditions on the model's
parameters that lead to the instability were later coined in \citep{Yi2008}.
The same authors also considered the global asymptotic behaviour of the
model in \citep{Yi2009}. Dynamics of the two-dimensional case where
discussed in \citep{Wang2013}. The findings of \citep{Ni2005} and %
\citep{Yi2009} in terms of the sufficient conditions for global asymptotic
stability were confirmed and extended in \citep{Lisena2014}.

Since the Lengyel--Epstein model only considers a single reaction, it is of
particular importance to generalise its reactions terms to encompass other
variations of the model. A first attempt to generalise the model was
achieved in \citep{Abdelmalek2017a}, where it was shown how other
approximations of the reaction term could be studied in a general way. The
same model was studied again in \citep{Abdelmalek2017b}, where the authors
estabished sufficient conditions for the non--existence of Turing patterns.
The authors also followed on the footsteps of \citep{Lisena2014} to relax
the global asymptotic stability conditions. Another study related to the
same model is \citep{Abdelmalek2017c}, whre the authors established the
boundedness of solutions.

This paper presents a broader generalisation of the Lengyel--Epstein model
to encompass many existing systems such as the FitzHugh--Nagumo model %
\citep{Oshita2003,Doelman2009} as will be shown later on in Section \ref%
{SecApp}. In Section \ref{SecModel}, we will present the general model
proposed in this paper. In the consequent sections, we will study the
dynamics of the proposed system.

\section{System Model\label{SecModel}}

In this paper, we consider the reaction--diffusion system%
\begin{equation}
\left \{ 
\begin{array}{l}
\frac{\partial u}{\partial t}-d_{1}\Delta u=\left( f\left( u\right) -\lambda
v\right) \varphi \left( u\right) :=F\left( u,v\right) \text{ \ \ \ in }%
\mathbb{R}^{+}\times \Omega ,\smallskip \\ 
\frac{\partial v}{\partial t}-d_{2}\Delta v=\sigma \left( g\left( u\right)
-v\right) \varphi \left( u\right) :=G\left( u,v\right) \text{ \ \ \ in }%
\mathbb{R}^{+}\times \Omega ,%
\end{array}%
\right.  \label{1.1}
\end{equation}%
where $\Omega $ is a bounded domain in $%
\mathbb{R}
^{n}$ with smooth boundary $\partial \Omega $ and $\Delta $ is the Laplacian
operator on $\Omega $. We assume non-negative continuous and bounded initial
data%
\begin{equation}
u\left( 0,x\right) =u_{0}\left( x\right) ,\text{ \ }v\left( 0,x\right)
=v_{0}\left( x\right) \text{ \ \ \ \ \ \ \ \ \ \ in }\Omega ,  \label{1.2}
\end{equation}%
where $u_{0}\left( x\right) ,v_{0}\left( x\right) \in C^{2}\left( \Omega
\right) \cap C\left( \overline{\Omega }\right) $, and homogoneous Neumann
boundary conditions 
\begin{equation}
\dfrac{\partial u}{\partial \nu }=\dfrac{\partial v}{\partial \nu }=0\ \ 
\text{\ \ \ \ \ on \ \ \ }\mathbb{R}^{+}\times \partial \Omega ,  \label{1.3}
\end{equation}%
with $\nu $ being the unit outer normal to $\partial \Omega $.

The constants $d_{1},d_{2},\lambda ,$ and $\sigma $ are assumed to be
strictly positive control parameters. Note that $d_{1}$ and $d_{2}$
represent the diffusivity constants, which in a chemical reaction are
proportional to the ratios between the molar flux and the concentration
gradient of the reactants. The functions\ $\varphi ,g$ and $f$ are assumed
to be continuously differentiable on $\mathbb{R}^{+}$ such that for some $%
\delta \in 
\mathbb{R}
^{+}$,

\begin{equation}
\varphi \left( 0\right) =0,f\left( \delta \right) =0,  \label{con1}
\end{equation}%
and for $u\in \left( 0,\delta \right) $,%
\begin{equation}
g(u),f\left(u\right),\varphi\left(u\right) >0,  \label{con5}
\end{equation}%
and%
\begin{equation}
g^{\prime }(u)\geq 0.  \label{con2}
\end{equation}%
We also suppose that there exists a positive constant $\alpha \in \left(
0,\delta \right) $ such that%
\begin{equation}
\lambda g\left( \alpha \right) =f\left( \alpha \right),  \label{con3}
\end{equation}%
and%
\begin{equation}
\left( \alpha -u\right) \left[ f\left( u\right) -\lambda g\left( u\right) %
\right] >0\text{ \ for\ }u\in \left( 0,\alpha \right) \cup \left( \alpha
,\delta \right) .  \label{con4}
\end{equation}

\section{Preliminaries}

In this section, we will present some preliminary results. First, we show
that subject to (\ref{con2}), the system has an invariant region. Then, we
identify a unique equilibrium solution for the ODE system and establish its
local asymptotic stability under certain conditions. Finally, the local
asymptotic stability of the steady state solution in the presence of
diffusions is established under some sufficient conditions.

\subsection{Invariant Regions}

In this subsection, we examine the invariant regions for the system (\ref%
{1.1}).

\begin{definition}[\citep{Yi2009,Mottoni1979}]
\label{DefInvariant}A rectangle $\Re =\left( 0,r_{1}\right) \times \left(
0,r_{2}\right) $ is called an invariant rectangle if the vector field $%
\left( F,G\right) $ on the boundary $\partial \Re $ points inside, that is,%
\begin{equation*}
\left \{ 
\begin{array}{c}
F\left( 0,v\right) \geq 0\text{ and }F\left( r_{1},v\right) \leq 0\text{ for 
}0<v<r_{2}, \\ 
G\left( u,0\right) \geq 0\text{ and }G\left( u,r_{2}\right) \leq 0\text{ for 
}0<u<r_{1}.%
\end{array}%
\right.
\end{equation*}
\end{definition}

\begin{proposition}
\label{PropAttract}The system (\ref{1.1}), with condition (\ref{con2}) has
the invariant region%
\begin{equation*}
\Re =\left( 0,\delta \right) \times \left( 0,g\left( \delta \right) \right) .
\end{equation*}
\end{proposition}

\begin{proof}
From condition (\ref{con1}), we obtain 
\begin{equation*}
F\left( 0,v\right) =\lim_{u\rightarrow 0^{+}}f\left( u\right) \varphi \left(
u\right) \geq 0.
\end{equation*}

Also, using (\ref{con1}) and (\ref{con3}), we conclude that 
\begin{equation*}
\left( f\left( \delta \right) -f\left( \alpha \right) \right) -\lambda
v-g\left( \alpha \right) =-\left( 1+\lambda \right) g\left( \alpha \right)
-\lambda v\leq 0.
\end{equation*}%
It follows that%
\begin{equation*}
F\left( \delta ,v\right) =\left[ \left( f\left( \delta \right) -f\left(
\alpha \right) \right) -\lambda \left( v-g\left( \alpha \right) \right) %
\right] \varphi \left( \delta \right) \leq 0,
\end{equation*}%
\begin{equation*}
G\left( u,0\right) =\sigma \varphi \left( u\right) \left( g\left( u\right)
-0\right) =\sigma \varphi \left( u\right) g\left( u\right) \geq 0,
\end{equation*}%
and%
\begin{equation*}
G\left( u,g\left( \delta \right) \right) =\sigma \varphi \left( u\right)
\left( g\left( u\right) -g\left( \delta \right) \right) \leq 0.
\end{equation*}%
This concludes the proof.
\end{proof}

\subsection{Equilibrium Solutions and ODE Stability}

This section studies the uniform equilibrium solutions of the
reaction--diffusion system (\ref{1.1}). In the absense of diffusion, the
system reduces to%
\begin{equation}
\left \{ 
\begin{array}{l}
\frac{du}{dt}=\varphi \left( u\right) \left( f\left( u\right) -\lambda
v\right) \text{ \ \ in }\mathbb{R}^{+},\smallskip \\ 
\frac{dv}{dt}=\frac{\sigma }{\lambda }\varphi \left( u\right) \left( \lambda
g\left( u\right) -\lambda v\right) \text{ \ \ in }\mathbb{R}^{+}.%
\end{array}%
\right.  \label{2.2}
\end{equation}

\begin{proposition}
\label{PropEquilibrium}The system (\ref{2.2}) has the unique constant steady
state solution%
\begin{equation}
\left( u^{\ast },v^{\ast }\right) =\left( \alpha ,g\left( \alpha \right)
\right) .  \label{2.2a}
\end{equation}%
If the inequality%
\begin{equation}
f^{\prime }\left( \alpha \right) <\min \left \{ \sigma ,\lambda g^{\prime
}\left( \alpha \right) \right \}  \label{2.3}
\end{equation}%
is satisfied, then the solution is a locally asymptotically stable
equilibrium for the system (\ref{2.2}).
\end{proposition}

\begin{proof}
An equilibrium solution $\left( u^{\ast },v^{\ast }\right) $ satisfies%
\begin{equation*}
\left \{ 
\begin{array}{l}
\left[ \left( f\left( u^{\ast }\right) -f\left( \alpha \right) \right)
-\lambda \left( v^{\ast }-g\left( \alpha \right) \right) \right] \varphi
\left( u^{\ast }\right) =0, \\ 
\sigma \left[ \left( g\left( u^{\ast }\right) -g\left( \alpha \right)
\right) -\left( v^{\ast }-g\left( \alpha \right) \right) \right] \varphi
\left( u^{\ast }\right) =0.%
\end{array}%
\right.
\end{equation*}%
It is easy to see that $\left( \alpha ,g\left( \alpha \right) \right) $ is
the solution to this system thanks to conditions (\ref{con3}) and (\ref{con4}%
). It remains now to study the local asymptotic stability of the solution.
The Jacobian matrix is%
\begin{equation}
J\left( u,v\right) =\left( 
\begin{array}{cc}
F_{u}\left( u,v\right) & F_{v}\left( u,v\right) \\ 
G_{u}\left( u,v\right) & G_{v}\left( u,v\right)%
\end{array}%
\right) ,  \label{Jacob1}
\end{equation}%
where 
\begin{equation*}
F_{v}=-\lambda \varphi \left( u\right) ,\  \ G_{v}=-\sigma \varphi \left(
u\right) ,
\end{equation*}%
\begin{equation*}
F_{u}=\left( \left[ \left( f\left( u\right) -f\left( \alpha \right) \right)
-\lambda \left( v-g\left( \alpha \right) \right) \right] \varphi ^{\prime
}\left( u\right) +f^{\prime }\left( u\right) \varphi \left( u\right) \right)
,
\end{equation*}%
and%
\begin{equation*}
G_{u}=\sigma \left( \left( g\left( u\right) -v\right) \varphi ^{\prime
}\left( u\right) +g^{\prime }\left( u\right) \varphi \left( u\right) \right)
.
\end{equation*}%
Evaluating these derivatives for the equilibrium solution yields%
\begin{equation*}
F_{v}\left( \alpha ,g\left( \alpha \right) \right) =-\lambda \varphi \left(
\alpha \right) ,
\end{equation*}%
\begin{equation*}
G_{v}\left( \alpha ,g\left( \alpha \right) \right) =-\sigma \varphi \left(
\alpha \right) ,
\end{equation*}%
\begin{eqnarray*}
F_{u}\left( \alpha ,g\left( \alpha \right) \right) &=&\left( \left[ \left(
f\left( \alpha \right) -f\left( \alpha \right) \right) -\lambda \left(
g\left( \alpha \right) -g\left( \alpha \right) \right) \right] \varphi
^{\prime }\left( \alpha \right) +f^{\prime }\left( \alpha \right) \varphi
\left( \alpha \right) \right) \\
&=&f^{\prime }\left( \alpha \right) \varphi \left( \alpha \right) ,
\end{eqnarray*}%
and%
\begin{equation*}
G_{u}\left( \alpha ,g\left( \alpha \right) \right) =\sigma g^{\prime }\left(
\alpha \right) \varphi \left( \alpha \right) .
\end{equation*}%
Consequently,%
\begin{equation*}
J\left( u^{\ast },v^{\ast }\right) =\left( 
\begin{array}{cc}
f^{\prime }\left( \alpha \right) \varphi \left( \alpha \right) & -\lambda
\varphi \left( \alpha \right) \\ 
\sigma g^{\prime }\left( \alpha \right) \varphi \left( \alpha \right) & 
-\sigma \varphi \left( \alpha \right)%
\end{array}%
\right) .
\end{equation*}%
As the trace given by%
\begin{equation*}
{\text{tr}}J\left( u^{\ast },v^{\ast }\right) =\left[ f^{\prime }\left(
\alpha \right) -\sigma \right] \varphi \left( \alpha \right) <0,
\end{equation*}

and the determinant given by%
\begin{equation}
\det J\left( u^{\ast },v^{\ast }\right) =\sigma \varphi ^{2}\left( \alpha
\right) \left[ \lambda g^{\prime }\left( \alpha \right) -f^{\prime }\left(
\alpha \right) \right] >0,
\end{equation}%
the equilibrium is then locally asymptotically stable.
\end{proof}

\begin{remark}
Observe that $F_{v}\left( u^{\ast },v^{\ast }\right) <0,$ $G_{v}\left(
u^{\ast },v^{\ast }\right) <0$ and\ $G_{u}\left( u^{\ast },v^{\ast }\right)
>0$. Recall that $\varphi \left( \alpha \right) $ is strictly positive.
Hence, if 
\begin{equation}
F_{u}\left( u^{\ast },v^{\ast }\right) =f^{\prime }\left( \alpha \right)
\varphi \left( \alpha \right) >0  \label{2.5}
\end{equation}%
is satisfied, then $u$ is called an activator, $v$ is called an inhibitor,
and the system (\ref{2.2}) is an activator--inhibitor system.
\end{remark}

\begin{remark}
Combining the activator-inhibitor condition\ (\ref{2.5}) with the stability
condition\ (\ref{2.3}), we find that the condition%
\begin{equation}
0<f^{\prime }\left( \alpha \right) <\min \left \{ \sigma ,\lambda g^{\prime
}\left( \alpha \right) \right \}  \label{2.6}
\end{equation}%
makes the model (\ref{2.2}) a diffusion-free stable activator-inhibitor
system.
\end{remark}

\subsection{PDE Stability}

Let us consider the local asymptotic stability of the steady state solutions
in the PDE case. Let $0=\lambda _{0}<\lambda _{1}\leq \lambda _{3}\leq ....$
be the sequence of eigenvalues for ($-\Delta $) subject to the Neumann
boundary conditions on $\Omega ,$ where each $\lambda _{i}$ has multiplicity 
$m_{i}\geq 1$. Also let $\Phi _{ij},1\leq j\leq m_{i}$, (recall that $\Phi
_{0}=const$ and $\lambda _{i}\rightarrow \infty $ at $i\rightarrow \infty $)
be the normalized eigenfunctions corresponding to $\lambda _{i}$. That is, $%
\Phi _{ij}$ and $\lambda _{i}$ satisfy $-\Delta \Phi _{ij}=\lambda _{i}\Phi
_{ij}$ in $\Omega $, with $\frac{\partial \Phi _{ij}}{\partial \nu }=0$ in $%
\partial \Omega $, and $\int_{\Omega }\Phi _{ij}^{2}\left( x\right) dx=1$.

The set $\left\{ \Phi _{ij}:i\geq 0,1\leq j\leq m_{i}\right\} $ forms a
complete orthonormal basis in $L^{2}\left( \Omega \right) $. If%
\begin{equation}
d_{1}\lambda _{1}<F_{0}:=f^{\prime }\left( \alpha \right) \varphi \left(
\alpha \right) ,  \label{2.9}
\end{equation}%
then we may define $i_{\alpha }=i_{\alpha }(\alpha ,\Omega )$ to be the
largest positive integer such that%
\begin{equation*}
d_{1}\lambda _{i}<F_{0}\text{ \ \ for all\ \ }i\leq i_{\alpha }.
\end{equation*}%
Clearly, if (\ref{2.9}) is satisfied, then $1\leq i_{\alpha }<\infty $. In
this case, we define 
\begin{equation}
d=d\left( \alpha ,\Omega \right) =\underset{1\leq i<i_{\alpha }}{\min }%
\widetilde{d}_{i},\text{ \ }\widetilde{d}_{i}=\varphi \left( \alpha \right) 
\frac{\lambda _{i}d_{1}+\varphi \left( \alpha \right) \left( \lambda
g^{\prime }\left( \alpha \right) -f^{\prime }\left( \alpha \right) \right) }{%
\lambda _{i}\left( F_{0}-\lambda _{i}d_{1}\right) }.  \label{2.10}
\end{equation}

\begin{proposition}
Subject to (\ref{2.6}), if either $\lambda _{1}d_{1}\geq F_{0}$ or $\lambda
_{1}d_{1}<F_{0}$ and $0<\frac{d_{2}}{\sigma }<d$, then the constant steady
state $\left( u^{\ast },v^{\ast }\right) $ is locally asymptotically stable.
Otherwise, if \ $\lambda _{1}d_{1}<F_{0}$ and $d<\frac{d_{2}}{\sigma }$,
then $\left( u^{\ast },v^{\ast }\right) $ is locally asymptotically unstable.
\end{proposition}

\begin{proof}
First, let us define the operator%
\begin{align*}
L& =\left( 
\begin{array}{cc}
d_{1}\Delta +f^{\prime }\left( \alpha \right) \varphi \left( \alpha \right)
& -\lambda \varphi \left( \alpha \right) \\ 
\sigma g^{\prime }\left( \alpha \right) \varphi \left( \alpha \right) & 
d_{2}\Delta -\sigma \varphi \left( \alpha \right)%
\end{array}%
\right) \\
& =\left( 
\begin{array}{cc}
d_{1}\Delta +F_{0} & F_{1} \\ 
\sigma G_{0} & d_{2}\Delta +\sigma G_{1}%
\end{array}%
\right) .
\end{align*}%
The steady state solution $\left( u^{\ast },v^{\ast }\right) $ is locally
assymptotically stable if all the eigenvalues of $L$ have negative real
parts, see for instance \citep{Casten1977}. On the contrary, if some
eigenvalues have positive real parts, then the steady state is locally
asymptotically unstable. We have%
\begin{equation*}
L\left( \phi \left( x\right) ,\psi \left( x\right) \right) ^{t}=\xi \left(
\phi \left( x\right) ,\psi \left( x\right) \right) ^{t},
\end{equation*}%
where $\left( \phi \left( x\right) ,\psi \left( x\right) \right) $ is an
eigenfunction of $L$ corresponding to an eigenvalue $\xi $; this can be
rearranged to%
\begin{equation*}
\left( 
\begin{array}{cc}
d_{1}\Delta +F_{0}-\xi & F_{1} \\ 
\sigma G_{0} & d_{2}\Delta +\sigma G_{1}-\xi%
\end{array}%
\right) \left( 
\begin{array}{c}
\phi \\ 
\psi%
\end{array}%
\right) =\left( 
\begin{array}{c}
0 \\ 
0%
\end{array}%
\right) ,
\end{equation*}%
which can be rewritten as%
\begin{equation*}
\sum_{0\leq i\leq \infty ,1\leq j\leq m_{i}}\left( 
\begin{array}{cc}
F_{0}-d_{1}\lambda _{i}-\xi & F_{1} \\ 
\sigma G_{0} & \sigma G_{1}-d_{2}\lambda _{i}-\xi%
\end{array}%
\right) \left( 
\begin{array}{c}
a_{ij} \\ 
b_{ij}%
\end{array}%
\right) \Phi _{ij}=\left( 
\begin{array}{c}
0 \\ 
0%
\end{array}%
\right) ,
\end{equation*}%
where%
\begin{equation*}
\phi =\sum_{0\leq i\leq \infty ,1\leq j\leq m_{i}}a_{ij}\Phi _{ij},
\end{equation*}%
and%
\begin{equation*}
\psi =\sum_{0\leq i\leq \infty ,1\leq j\leq m_{i}}b_{ij}\Phi _{ij}.
\end{equation*}%
Observe that 
\begin{equation*}
\det \left( 
\begin{array}{cc}
F_{0}-d_{1}\lambda _{i}-\xi & F_{1} \\ 
\sigma G_{0} & \sigma G_{1}-d_{2}\lambda _{i}-\xi%
\end{array}%
\right) =\xi ^{2}+p_{i}\xi +Q_{i},
\end{equation*}%
with%
\begin{align*}
p_{i}& =\left( \lambda _{i}d_{1}-\sigma G_{1}-F_{0}+\lambda _{i}d_{2}\right)
\\
& =\left( d_{1}+d_{2}\right) \lambda _{i}+\left[ \sigma  -f^{\prime }\left( \alpha \right) \right] \varphi \left( \alpha
\right) >0,
\end{align*}%
by condition (\ref{2.6}), and%
\begin{align*}
Q_{i}& =\lambda _{i}^{2}d_{1}d_{2}+\sigma F_{0}G_{1}-\lambda
_{i}F_{0}d_{2}-\sigma \lambda _{i}G_{1}d_{1}-\sigma F_{1}G_{0} \\
& =\sigma \left[ \lambda _{i}\frac{d_{2}}{\sigma }\left( \lambda
_{i}d_{1}-F_{0}\right) -G_{1}\lambda _{i}d_{1}+F_{0}G_{1}-F_{1}G_{0}\right]
\\
& =\sigma \left[ \frac{d_{2}}{\sigma }\lambda _{i}\left( \lambda
_{i}d_{1}-F_{0}\right) +\varphi \left( \alpha \right) \left \{ \lambda
_{i}d_{1}+\varphi \left( \alpha \right) \left( \lambda g^{\prime }\left(
\alpha \right) -f^{\prime }\left( \alpha \right) \right) \right \} \right] ;
\end{align*}%
note that $Q_{0}>0$ for $\lambda _{0}=0$. Hence, one may easily observe that 
$\xi $ is an eigenvalue of $L$ iff for some $i\geq 0$,%
\begin{equation*}
\xi ^{2}+p_{i}\xi +Q_{i}=0.
\end{equation*}

We can study the three cases stated in the proposition above separately:

\begin{enumerate}
\item If $\lambda _{1}d_{1}\geq F_{0}$, then $Q_{i}>0$ for $i\geq 1$. The
fact that for $i\geq 0$, both $p_{i}>0$ and $Q_{i}>0$ for $i\geq 0$ implies
that $\text{Re}\xi <0$ for all eigenvalues $\xi $; consequently the steady
state $\left( u^{\ast },v^{\ast }\right) $ is locally asymptotically stable.

\item We consider the case where $\lambda _{1}d_{1}<F_{0}$ and $0<\frac{d_{2}%
}{\sigma }<d$, which leads to%
\begin{equation*}
\lambda _{1}d_{1}<F_{0}\text{ and }0<\frac{d_{2}}{\sigma }<\widetilde{d}_{i},
\end{equation*}%
for $i\in \left[ 1,i_{\alpha }\right] $. It simply follows that $Q_{i}>0$
for $i\in \left[ 1,i_{\alpha }\right] $. Furthermore, if $i\geq i_{\alpha }$%
, then $\lambda _{i}d_{1}\geq F_{0}$ and $Q_{i}>0$; this leads to the local
asymptotic stability of $\left( u^{\ast },v^{\ast }\right) $ again.

\item If $\lambda _{1}d_{1}<F_{0}$ and $d<\frac{d_{2}}{\sigma }$, then we
may assume that the minimum in (\ref{2.10}) is reached by $k\in \left[
1,i_{\alpha }\right] $. Thus,%
\begin{equation}
\frac{d_{2}}{\sigma }>\widetilde{d}_{k},  \label{dCond}
\end{equation}%
which implies $Q_{k}<0$, and consequently the instability of $\left( u^{\ast
},v^{\ast }\right) $ follows.
\end{enumerate}
\end{proof}

\section{Boundedness of Solutions}

In this section, we would like to establish the global existence of
solutions for the system (\ref{1.1}). We start by proving that it has a
unique solution $\left( u\left( x,t\right) ,v\left( x,t\right) \right) $ for
all $x\in \Omega $ and $t>0$, which is bounded by some positive constants
depending on $u_{0}$ and $v_{0}$. In order to establish the boundedness of
solutions, it is assumed that $\varphi \ $is a sublinear function, i.e. the
mapping $(0,\infty )\ni s\rightarrow \frac{\varphi (s)}{s}$ is
non--increasing. In a similar manner to Lemma 1 of \citep{Abdelmalek2017a},
we may establish that the sublinearity of $\varphi $ along with the first
part of (\ref{con1}), $\varphi \left( 0\right) =0$, gives%
\begin{equation}
0<\frac{\varphi \left( u\right) }{u}\leq \varphi ^{\prime }\left( 0\right) .
\label{02.1}
\end{equation}%
Let us also assume that%
\begin{equation*}
f\left( u\right) \varphi \left( u\right) =K-u\Psi \left( u\right) ,
\end{equation*}%
with $K$ being a positive real number and $\Psi \left( u\right) $ a positive
bounded function. Similarly, we assume that%
\begin{equation*}
g\left( u\right) \varphi \left( u\right) =u\Phi \left( u\right) ,
\end{equation*}%
where $\Phi \left( u\right) $ is a positive bounded function. The following
propositions are based on the work of Ni and Tang (2005) \citep{Ni2005} for
the original Lengyel--Epstein system.

\begin{proposition}
\label{Prop1}The system (\ref{1.1}) admits a unique solution $(u,v)$ for all 
$x\in \Omega $ and $t>0$ and there exist two positive constants $C_{1}$ and $%
C_{2}$ such that%
\begin{equation}
C_{1}<u\left( x,t\right) ,v\left( x,t\right) <C_{2}.  \label{02.2}
\end{equation}
\end{proposition}

\begin{proof}
Since the local existence and uniqueness of solutions are classical for the
proposed system, see \citep{Friedman1964}, it suffices to establish the
global existence by proving the boundedness of the solution. To this aim, we
will use the invariant regions theory as proposed in \citep{Weinberger1975}.
We take a certain rectangular region of the form%
\begin{equation*}
R=\left( u_{1},u_{2}\right) \times \left( v_{1},v_{2}\right) ,
\end{equation*}%
and study the behavior of the vector field along its four boundaries
separately.

\begin{itemize}
\item On the left boundary of $R$, we have $u=u_{1}$ and $v_{1}\leq v\leq
v_{2}$. Hence,%
\begin{eqnarray*}
F\left( u,v\right) &=&\left( f\left( u\right) -\lambda v\right) \varphi
\left( u\right) \\
&\geq &f\left( u_{1}\right) \varphi \left( u_{1}\right) -\lambda
v_{2}\varphi \left( u_{1}\right) \\
&\geq &K-u_{1}\Psi \left( u_{1}\right) -\lambda v_{2}\varphi \left(
u_{1}\right) \\
&\geq &K-u_{1}\left[ \Psi \left( u_{1}\right) +\lambda v_{2}\frac{\varphi
\left( u_{1}\right) }{u_{1}}\right] \\
&\geq &K-u_{1}\left[ \Psi \left( u_{1}\right) +\lambda v_{2}\varphi ^{\prime
}\left( 0\right) v_{2}\right] .
\end{eqnarray*}%
A sufficient condition for $F\left( u,v\right) \geq 0$ can then be
formulated as%
\begin{equation*}
K-u_{1}\left[ \Psi \left( u_{1}\right) +\lambda v_{2}\varphi ^{\prime
}\left( 0\right) v_{2}\right] \geq 0.
\end{equation*}%
Whereupon%
\begin{equation*}
u_{1}\leq \frac{K}{\Psi _{\min }+\lambda v_{2}\varphi ^{\prime }\left(
0\right) }.
\end{equation*}

\item For the right boundary where $u=u_{2}$ and $v_{1}\leq v\leq v_{2}$, we
have%
\begin{eqnarray*}
F\left( u,v\right) &=&f\left( u\right) \varphi \left( u\right) -\lambda
v\varphi \left( u\right) \\
&\leq &f\left( u_{2}\right) \varphi \left( u_{2}\right) \\
&\leq &K-u_{2}\Psi \left( u_{2}\right) .
\end{eqnarray*}%
It suffices that%
\begin{equation*}
K-u_{2}\Psi \left( u_{2}\right) \leq 0,
\end{equation*}%
or simply%
\begin{equation*}
\frac{K}{\Psi _{\max }}\leq u_{2},
\end{equation*}%
to guarantee the inquality $F\left( u,v\right) \leq 0$.
\end{itemize}

These two conditions yield the first part of the invariant region $R$:%
\begin{equation}
u_{1}=\min \left \{ \frac{K}{\Psi _{\min }+\lambda v_{2}\varphi ^{\prime
}\left( 0\right) },\min u_{0}\right \} ,  \label{02.3}
\end{equation}%
and%
\begin{equation}
u_{2}=\max \left \{ \frac{K}{\Psi _{\max }},\max u_{0}\right \} .
\label{02.4}
\end{equation}

\begin{itemize}
\item For the third boundary of $R$ where $v=v_{1}$ and $u_{1}\leq u\leq
u_{2}$,%
\begin{eqnarray*}
G\left( u,v\right) &=&\sigma \left( g\left( u\right) \varphi \left( u\right)
-v_{1}\varphi \left( u\right) \right) \\
&=&\sigma \left( u\Phi \left( u\right) -v_{1}\varphi \left( u\right) \right)
\\
&\geq &\sigma u\left( \Phi \left( u\right) -v_{1}\frac{\varphi \left(
u\right) }{u}\right) \\
&\geq &\sigma u\left( \Phi _{\min }-v_{1}\varphi ^{\prime }\left( 0\right)
\right) .
\end{eqnarray*}%
Hence, to achieve $G\left( u,v\right) \geq 0$, it suffices that%
\begin{equation*}
v_{1}\leq \frac{\Phi _{\min }}{\varphi ^{\prime }\left( 0\right) }.
\end{equation*}

\item For the boundary with $v=v_{2}$ and $u_{1}\leq u\leq u_{2}$,%
\begin{eqnarray*}
G\left( u,v\right) &=&\sigma \left( g\left( u\right) \varphi \left( u\right)
-v_{2}\varphi \left( u\right) \right) \\
&=&\sigma \left( u\Phi \left( u\right) -v_{2}\varphi \left( u\right) \right)
\\
&=&\sigma u\left( \Phi \left( u\right) -v_{2}\frac{\varphi \left( u\right) }{%
u}\right) \\
&\leq &\sigma u\left( \Phi _{\max }-v_{2}\frac{\varphi \left( u_{2}\right) }{%
u_{2}}\right) .
\end{eqnarray*}%
To ensure the negativity of $G\left( u,v\right) $ on this boundary, it is
sufficient to choose $v_{2}$ such that%
\begin{equation*}
v_{2}\geq \frac{u_{2}}{\varphi \left( u_{2}\right) }\Phi _{\max }.
\end{equation*}
\end{itemize}

Combining the conditions for these two boundaries gives us%
\begin{equation}
v_{1}=\min \left \{ \frac{\Phi _{\min }}{\varphi ^{\prime }\left( 0\right) }%
,\min v_{0}\right \} ,  \label{02.5}
\end{equation}%
and%
\begin{equation}
v_{2}=\max \left \{ \frac{u_{2}}{\varphi \left( u_{2}\right) }\Phi _{\max
},\max v_{0}\right \} .  \label{02.6}
\end{equation}

We can now simply define the bounds of the solutions as%
\begin{equation}
C_{1}=\min \left \{ u_{1},v_{1}\right \} >0,  \label{02.7}
\end{equation}%
and%
\begin{equation}
C_{2}=\min \left \{ u_{2},v_{2}\right \} >0.  \label{02.8}
\end{equation}
\end{proof}

\section{Global Asymptotic Stability}

We now pass to the global asymptotic stability for the system (\ref{1.1}).
For the global asymptotic stability of the steady state solution, we
consider the condition%
\begin{equation}
\left( \alpha -u\right) \left[ f\left( u\right) -f\left( \alpha \right) %
\right] >0\text{ \ for }u\in \left( 0,\alpha \right) \cup \left( \alpha
,\delta \right) ,  \label{con6}
\end{equation}%
which is clearly stronger than (\ref{con4}). System (\ref{1.1}) can now be
rewritten as%
\begin{equation*}
\left \{ 
\begin{array}{l}
\frac{\partial u}{\partial t}-d_{1}\Delta u=\left[ \left( f\left( u\right)
-f\left( \alpha \right) \right) -\lambda \left( v-g\left( \alpha \right)
\right) \right] \varphi \left( u\right) \text{ \ \ \ in }\mathbb{R}%
^{+}\times \Omega ,\smallskip \\ 
\frac{\partial v}{\partial t}-d_{2}\Delta v=\sigma \left[ \left( g\left(
u\right) -g\left( \alpha \right) \right) -\left( v-g\left( \alpha \right)
\right) \right] \varphi \left( u\right) \text{ \ \ \ in }\mathbb{R}%
^{+}\times \Omega .%
\end{array}%
\right.
\end{equation*}

As a first step, we start by establishing the conditions for the global
stabiliy of $\left( u^{\ast },v^{\ast }\right) $ as a solution of the
reduced ODE system%
\begin{equation}
\left \{ 
\begin{array}{l}
\frac{\partial u}{\partial t}=F\left( u,v\right) ,\smallskip \\ 
\frac{\partial v}{\partial t}=G\left( u,v\right) .%
\end{array}%
\right.  \label{3.1.0}
\end{equation}

\begin{theorem}
\label{Theo1}If for all $u$\ $\in \left( 0,\delta \right) $,%
\begin{equation}
f^{\prime }\left( u\right) <\sigma ,  \label{3.1.1}
\end{equation}%
then $\left( u^{\ast },v^{\ast }\right) $ is globally asymptotically stable.
\end{theorem}

\begin{proof}
System (\ref{3.1.0}) may be written as%
\begin{equation*}
\left \{ 
\begin{array}{l}
u_{t}=\varphi \left( u\right) \left[ f\left( u\right) -\lambda v\right] , \\ 
v_{t}=\sigma \varphi \left( u\right) \left( g\left( u\right) -v\right) .%
\end{array}%
\right.
\end{equation*}

Let us consider the vector field%
\begin{equation*}
\Psi \left( u,v\right) =\left( f\left( u\right) -\lambda v,\sigma g\left(
u\right) -\sigma v\right) ,
\end{equation*}

along with its divergence%
\begin{align*}
{\text{div}}\Psi \left( u,v\right) & =\frac{\partial }{\partial u}\left(
f\left( u\right) -\lambda v\right) +\frac{\partial }{\partial v}\left(
\sigma g\left( u\right) -\sigma v\right) \\
& =f^{\prime }\left( u\right) -\sigma .
\end{align*}%
Let us also consider the open rectangle $\Re $ with closure $\overline{\Re }$%
. We aim to show that%
\begin{equation}
\min_{\left( u,v\right) \in \overline{\Re }}\left( {\text{div}}\Psi \left(
u,v\right) \right) <0.  \label{3.11}
\end{equation}%
By combining (\ref{3.1.1}) and (\ref{3.11}), it becomes clear that%
\begin{equation*}
{\text{div}}\Psi \left( u,v\right) <0\text{ in }\Re .
\end{equation*}%
Therefore, making use of the classical Bendixson--Dulac criterion \citep%
{Burton1985}, system (\ref{3.1.0}) does not admit any periodic solutions in $%
\Re $. It follows from the Poincar\'{e}--Bendixson theorem \citep{Burton1985}
that for any solution $\left( u\left( t\right) ,v\left( t\right) \right) $
to (\ref{3.1.0}), the equality%
\begin{equation*}
\lim_{t\rightarrow \infty }\left \vert u\left( t\right) -u^{\ast }\right
\vert =0=\lim_{t\rightarrow \infty }\left \vert v\left( t\right) -v^{\ast
}\right \vert
\end{equation*}%
holds. This concludes the proof.
\end{proof}

\begin{theorem}
\label{TheoGlobal}If condition (\ref{con6}) is satisfied, then\ for any
solution $\left( u,v\right) $ to (\ref{1.1}), we get 
\begin{equation}
\lim_{t\rightarrow \infty }\left \Vert u\left( .,t\right) -u^{\ast }\right
\Vert _{L^{2}\left( \Omega \right) }=\lim_{t\rightarrow \infty }\left \Vert
v\left( .,t\right) -v^{\ast }\right \Vert _{L^{2}\left( \Omega \right) }=0.
\label{4.1}
\end{equation}
\end{theorem}

\begin{lemma}
\label{Lemma1}If $u\in \left( 0,\delta \right) $, then there exists a
constant $\gamma $ between $u$ and $\alpha $ such that%
\begin{equation*}
g\left( u\right) -g\left( \alpha \right) =\left( u-\alpha \right) g^{\prime
}\left( \gamma \right) .
\end{equation*}
\end{lemma}

\begin{lemma}
Consider the function $H$ defined as%
\begin{equation}
H\left( u\right) =\int_{\alpha }^{u}\left( g\left( r\right) -g\left( \alpha
\right) \right) dr  \label{4.2}
\end{equation}%
It follows that%
\begin{equation*}
H\left( u\right) \geq 0,\text{ \ and }\frac{d}{du}H\left( u\right) =g\left(
u\right) -g\left( \alpha \right) .
\end{equation*}
\end{lemma}

\begin{proof}
As a result of using (\ref{Lemma1}), there exists a $\gamma \left( r\right) $
in the interval $\left( \min \left \{ r,\alpha \right \} ,\max \left \{
r,\alpha \right \} \right) $ with $r$ lying between $u$\ and\ $\alpha $\
such that%
\begin{equation*}
H\left( u\left( x,t\right) \right) =\int_{\alpha }^{u}\left( r-\alpha
\right) g^{\prime }\left( \gamma \right) dr,
\end{equation*}%
where $\psi \left( \gamma \right) \geq 0$, leading to%
\begin{equation*}
\inf_{r\in \left( \min \left \{ u,\alpha \right \} ,\max \left \{ u,\alpha
\right \} \right) }g^{\prime }\left( \gamma \right) \int_{\alpha }^{u}\left(
r-\alpha \right) dr\leq H\left( u\left( x,t\right) \right) .
\end{equation*}%
Whereupon%
\begin{equation*}
0\leq \inf_{r}g^{\prime }\left( \gamma \right) \frac{1}{2}\left( u-\alpha
\right) ^{2}\leq H\left( u\left( x,t\right) \right) ,
\end{equation*}%
which shows that $H\left( u\right) \geq 0$.
\end{proof}

\begin{proposition}
\label{PropLyapunov}Let $\left( u\left( t,.\right) ,v\left( t,.\right)
\right) $ be a solution of (\ref{1.1})-(\ref{1.3}) and let%
\begin{equation}
V\left( t\right) =\int_{\Omega }E\left( u\left( x,t\right) ,v\left(
x,t\right) \right) dx,  \label{4.5}
\end{equation}%
where%
\begin{equation}
E\left( u,v\right) =\sigma H\left( u\right) +\frac{\lambda }{2}\left(
v-v^{\ast }\right) ^{2},  \label{4.6}
\end{equation}%
then subject to (\ref{con6}),\ $V\left( t\right) $\ is a Lyapunov functional.
\end{proposition}

\begin{proof}
First of all,%
\begin{equation*}
V\left( t\right) =\int_{\Omega }\left[ \sigma H\left( u\right) +\frac{%
\lambda }{2}\left( v-v^{\ast }\right) ^{2}\right] dx.
\end{equation*}%
We have%
\begin{align*}
\frac{d}{dt}V\left( t\right) & =\frac{d}{dt}\int_{\Omega }\left[ \sigma
H\left( u\right) +\frac{\lambda }{2}\left( v-v^{\ast }\right) ^{2}\right] dx
\\
& =\int_{\Omega }\sigma \frac{d}{dt}\left[ H\left( u\right) \right] dx+\frac{%
\lambda }{2}\frac{d}{dt}\int_{\Omega }\left( v-v^{\ast }\right) ^{2}dx \\
& =\sigma \int_{\Omega }\left( g\left( u\right) -g\left( u^{\ast }\right)
\right) u_{t}dx+\lambda \int_{\Omega }\left( v-v^{\ast }\right) v_{t}dx \\
& =\sigma \int_{\Omega }\left( g\left( u\right) -g\left( u^{\ast }\right)
\right) \left( d_{1}\Delta u+\varphi \left( u\right) \left[ \left( f\left(
u\right) -f\left( u^{\ast }\right) \right) -\lambda \left( v-v^{\ast
}\right) \right] \right) dx \\
& \  \  \  \  \  \  \  \ +\lambda \int_{\Omega }\left( v-v^{\ast }\right) \left(
d_{2}\Delta u+\sigma \varphi \left( u\right) \left[ \left( g\left( u\right)
-g\left( u^{\ast }\right) \right) -\left( v-v^{\ast }\right) \right] \right)
dx.
\end{align*}%
For reasons that will become clear at the end of this proof, let us set%
\begin{equation*}
I=\sigma d_{1}\int_{\Omega }\left( g\left( u\right) -g\left( u^{\ast
}\right) \right) \Delta u\,dx+\lambda d_{2}\int_{\Omega }\left( v-v^{\ast
}\right) \Delta v\,dx,
\end{equation*}%
and%
\begin{equation*}
J=\int_{\Omega }\sigma \varphi \left( u\right) \left[ \left( g\left(
u\right) -g\left( u^{\ast }\right) \right) \left( f\left( u\right) -f\left(
u^{\ast }\right) \right) -\lambda \left( v-v^{\ast }\right) ^{2}\right] dx.
\end{equation*}

Let us also set%
\begin{equation*}
I=I_{1}+I_{2},
\end{equation*}%
where%
\begin{align*}
I_{1}& =\sigma d_{1}\int_{\Omega }\left( g\left( u\right) -g\left( u^{\ast
}\right) \right) \Delta u\,dx \\
& =-\sigma d_{1}\int_{\Omega }g^{\prime }\left( u\right) \left \vert \nabla
u\right \vert ^{2}\,dx,
\end{align*}%
and%
\begin{align*}
I_{2}& =\lambda d_{2}\int_{\Omega }\left( v-v^{\ast }\right) \Delta v\,dx \\
& =-\lambda d_{2}\int_{\Omega }\left \vert \nabla v\right \vert ^{2}\,dx,
\end{align*}%
in the light of Green's formula ($\dfrac{\partial u}{\partial \nu }=\dfrac{%
\partial v}{\partial \nu }=0$); whereupon%
\begin{equation*}
I=-\sigma d_{1}\int_{\Omega }g^{\prime }\left( u\right) \left \vert \nabla
u\right \vert ^{2}dx-\lambda d_{2}\int_{\Omega }\left \vert \nabla v\right
\vert ^{2}dx.
\end{equation*}

Therefore, under assumption (\ref{con2}) it follows that $I\leq0$.

Now, let us examine%
\begin{equation}
J=\int_{\Omega }\sigma \varphi \left( u\right) \left[ \left( g\left(
u\right) -g\left( u^{\ast }\right) \right) \left( f\left( u\right) -f\left(
u^{\ast }\right) \right) -\lambda \left( v-v^{\ast }\right) ^{2}\right] dx.
\label{4.8}
\end{equation}%
We have%
\begin{equation*}
J=\int_{\Omega }\sigma \varphi \left( u\right) \left[ g^{\prime }\left(
\gamma _{2}\right) \left( u-u^{\ast }\right) \left( f\left( u\right)
-f\left( u^{\ast }\right) \right) -\lambda \left( v-v^{\ast }\right) ^{2}%
\right] dx.
\end{equation*}%
If condition (\ref{con6}) is satisfied, then%
\begin{eqnarray*}
u &\leq &u^{\ast }\implies \left( u-u^{\ast }\right) \left( f\left( u\right)
-f\left( u^{\ast }\right) \right) \leq 0, \\
u &\geq &u^{\ast }\implies \left( u-u^{\ast }\right) \left( f\left( u\right)
-f\left( u^{\ast }\right) \right) \leq 0.
\end{eqnarray*}

It is easy to see that $J\leq 0$, and therefore%
\begin{equation*}
\frac{d}{dt}V\left( t\right) \leq 0.
\end{equation*}%
This concludes the proof of the proposition.
\end{proof}

\begin{proof}
\lbrack Proof of Theorem 1] The positive-definite functional $V\left(
t\right) $ has a non-positive derivative. Moreover, if $\left( u\left(
x,t\right) ,v\left( x,t\right) \right) \in \Re $ is a solution of (\ref{1.1}%
), for which $\frac{d}{dt}V\left( t\right) =0$, it follows necessarily that $%
\left \vert \nabla u\right \vert ^{2}=\left \vert \nabla v\right \vert ^{2}=0
$; that is $u$ and $v$ are spatially homogeneous. Hence, $\left( u,v\right) $
satisfies the ODE system (\ref{2.2}). Since, for the differential system (%
\ref{2.2}), $\left( u^{\ast },v^{\ast }\right) $ is the largest invariant
subset%
\begin{equation*}
\left \{ \left( u\left( x,t\right) ,v\left( x,t\right) \right) \in \Re \mid 
\frac{d}{dt}V\left( t\right) =0\right \} ,
\end{equation*}%
one gets (see \citep{Lisena2014, Yi2009}) via La Salle's invariance theorem%
\begin{equation*}
\lim_{t\rightarrow \infty }\left \vert u\left( x,t\right) -u^{\ast }\right
\vert =\lim_{t\rightarrow \infty }\left \vert v\left( x,t\right) -v^{\ast
}\right \vert =0,
\end{equation*}%
uniformly in $x$. Hence,%
\begin{equation}
\lim_{t\rightarrow \infty }\int_{\Omega }\left( u-u^{\ast }\right)
^{2}\left( x,t\right) =\lim_{t\rightarrow \infty }\int_{\Omega }\left(
v-v^{\ast }\right) ^{2}\left( x,t\right) =0.  \label{4.10}
\end{equation}
\end{proof}

\section{A Remark\label{Phi0Remark}}

Recall that condition (\ref{con1}) was required for Proposition \ref%
{PropAttract} to hold. However, it can be shown that if $\varphi \left(
0\right) >0$ and the inequality%
\begin{equation}
\lambda g\left( \delta \right) \leq \underset{u\rightarrow 0^{+}}{\lim }%
f\left( u\right)  \label{PhiGen}
\end{equation}%
is fulfilled, then the proposition still holds.

\section{Applications\label{SecApp}}

In this section, we will present two concrete examples that can be
considered special cases of system (\ref{1.1}). The two examples were
deliberately selected to cover two separate cases for $\varphi \left(
0\right) $; the first being $\varphi \left( 0\right) =0$ as stated in
condition (\ref{con1}), and the second being $\varphi \left( 0\right) >0$ as
stated in Subsection \ref{Phi0Remark}. For the two chosen examples, we will
apply the findings of this study to establish the global existence of
solutions and show that under the previously imposed conditions, the systems
are globally asymptotically stable.

\subsection{Lengyel--Epstein Model (CIMA\ Reaction)}

Consider the case%
\begin{equation}
f\left( u\right) =\frac{a-\mu u}{\varphi \left( u\right) }\text{ and }%
g\left( u\right) =\frac{u}{\varphi \left( u\right) },  \label{App1}
\end{equation}%
for $u\in \left( 0,\delta \right] $, with condition%
\begin{equation*}
\frac{u-\alpha }{\frac{a}{\mu }-u}\varphi \left( u\right) \geq \frac{%
u-\alpha }{\frac{a}{\mu }-\alpha }\varphi \left( \alpha \right) ,\text{ \ \ }%
u\in \left( 0,\frac{a}{\mu }\right) ,
\end{equation*}%
which is a special case of (\ref{con6}). Therefore, the example in (\ref%
{App1}) satisfies the conditions set out in this work. Substituting (\ref%
{App1}) in system (\ref{1.1}) yields%
\begin{equation}
\left \{ 
\begin{array}{l}
\frac{\partial u}{\partial t}-d_{1}\Delta u=a-\mu u-\lambda \varphi \left(
u\right) v,\text{ \ \ \ in }\mathbb{R}^{+}\times \Omega ,\smallskip \\ 
\frac{\partial v}{\partial t}-d_{2}\Delta v=\sigma \left( u-\varphi \left(
u\right) v\right) ,\text{ \ \ \ \ \ \ \ \ in }\mathbb{R}^{+}\times \Omega .%
\end{array}%
\right.  \label{App1-1}
\end{equation}%
It is easy to see that the resulting system (\ref{App1-1}) is the same as
the generalized Lengyel-Epstein system proposed in \citep{Abdelmalek2017a}.
Now, from condition (\ref{con3}), we have%
\begin{equation*}
f\left( \alpha \right) =\lambda g\left( \alpha \right) ,
\end{equation*}%
which gives%
\begin{equation*}
\alpha =\frac{a}{\lambda +\mu }.
\end{equation*}%
The equilibrium solution of the system is%
\begin{equation*}
\left( u^{\ast },v^{\ast }\right) =\left( \alpha ,\frac{\alpha }{\varphi
\left( \alpha \right) }\right) ,
\end{equation*}%
where%
\begin{equation*}
\alpha =\frac{a}{\lambda +\mu }.
\end{equation*}

For instance, let us suppose that $d_{1}=1,$ $\varphi \left( u\right) =\frac{%
u}{1+u^{2}},$ $\lambda =4,$ $\mu =1,$ and $d_{2}=\sigma c$. Also, for
modelling purposes, we will replace the positive constant $\sigma $ with the
product$\ \sigma b$. Substituting these parameters in system (\ref{1.1})
yields the original Lengyel--Epstein system \citep{Lengyel1991,Lengyel1992}%
\begin{equation}
\left \{ 
\begin{array}{l}
\frac{\partial u}{\partial t}=\Delta u+a-u-\frac{4uv}{1+u^{2}},\medskip \\ 
\frac{\partial v}{\partial t}=\left( \sigma c\right) \Delta v+\left( \sigma
b\right) \left( u-\frac{uv}{1+u^{2}}\right) .%
\end{array}%
\right.  \label{Ex1}
\end{equation}%
The system (\ref{Ex1}) represents DeKepper's chlorite--iodide--malonic
acid--starch chemical experiment \citep{DeKepper1990}, which was the first
ever realisation of Turing's instability \citep{Turing1952}.

The dynamics of the Lengyel--Epstein system (\ref{Ex1}) have been deeply
studied in the literature and thus will not be shown here. It suffices to
determine the range of $a$ for which the solutions are guaranteed to be
asymptotically stable. We notice that according to condition (\ref{con6}), $%
f(u)$ is decreasing regardless of $a$ for $u\in \left( 0,\alpha \right] $,
thus fulfilling the stability condition. For $u\in \left[ \alpha ,\delta
\right) $, the function $f(u)$ remains below $f(u^{\ast })$ as long as%
\begin{equation*}
a^{2}\leq \frac{125}{4}.
\end{equation*}%
This same result was achieved in \citep{Lisena2014} for the original
Lengyel--Epstein model, and in \citep{Abdelmalek2017b}\ for the generalised
system. The functions $f\left( u\right) $ and $\lambda g\left( u\right) $
are depicted in Figure \ref{Lengyel_Pap6_fg} for $a^{2}=\frac{125}{4}$ where 
$\alpha =\frac{a}{5}=1.118$. We note that for $a^{2}>\frac{125}{4}$, $f(u)$
rises above the horizontal line, thus not satisfying condition (\ref{con6}).
This, of course, does not necessarily imply that the solutions are not
globally asymptotically stable.

\begin{figure}[tbp]
\centering \includegraphics[width = 4.5in]{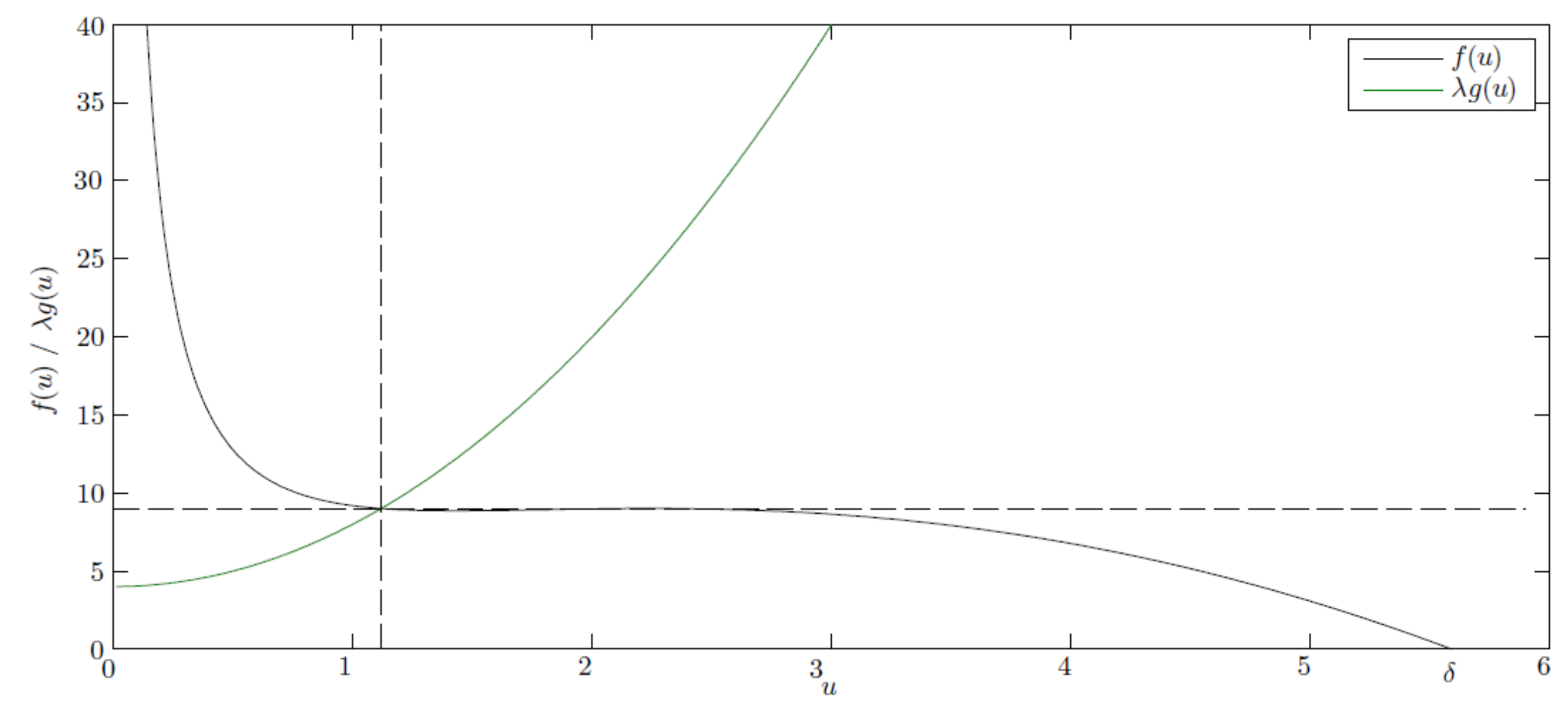}
\caption{The shape and intersection of functions $f(u)$ and $\protect\lambda %
g(u)$ for the Lengyel--Epstein reaction--diffusion model (\protect\ref{Ex1})
with maximum $a$.}
\label{Lengyel_Pap6_fg}
\end{figure}

The solutions of the Lengyel--Epstein model for $a^{2}=\frac{125}{4}$, $b=1$%
, $c=1$, and $\sigma =0.5$, are depicted in Figures \ref{Lengyel_Pap6_1}\
and \ref{Lengyel_Pap6_2}\ for the ODE and one--dimensional PDE cases,
respectively. The initial data is assumed to be%
\begin{equation*}
u\left( 0\right) =4\text{ and }v\left( 0\right) =3,
\end{equation*}%
in the ODE case and a slight sinusoidal perturbation is added in the
one-dimensional diffusion case%
\begin{equation*}
u\left( x,0\right) =4+0.2\sin \left( \frac{x}{5}\right) \text{ and }v\left(
x,0\right) =3+0.2\cos \left( \frac{x}{5}\right) .
\end{equation*}%
As expected from our analysis, we observe that the system is aymptotically
stable with the equilibrium solution given by (\ref{2.2a}) as%
\begin{equation*}
\left( u^{\ast },v^{\ast }\right) =\left( \alpha ,g\left( \alpha \right)
\right) =\left( 1.118,2.2499\right) .
\end{equation*}

\begin{figure}[tbp]
\centering \includegraphics[width = 4.5in]{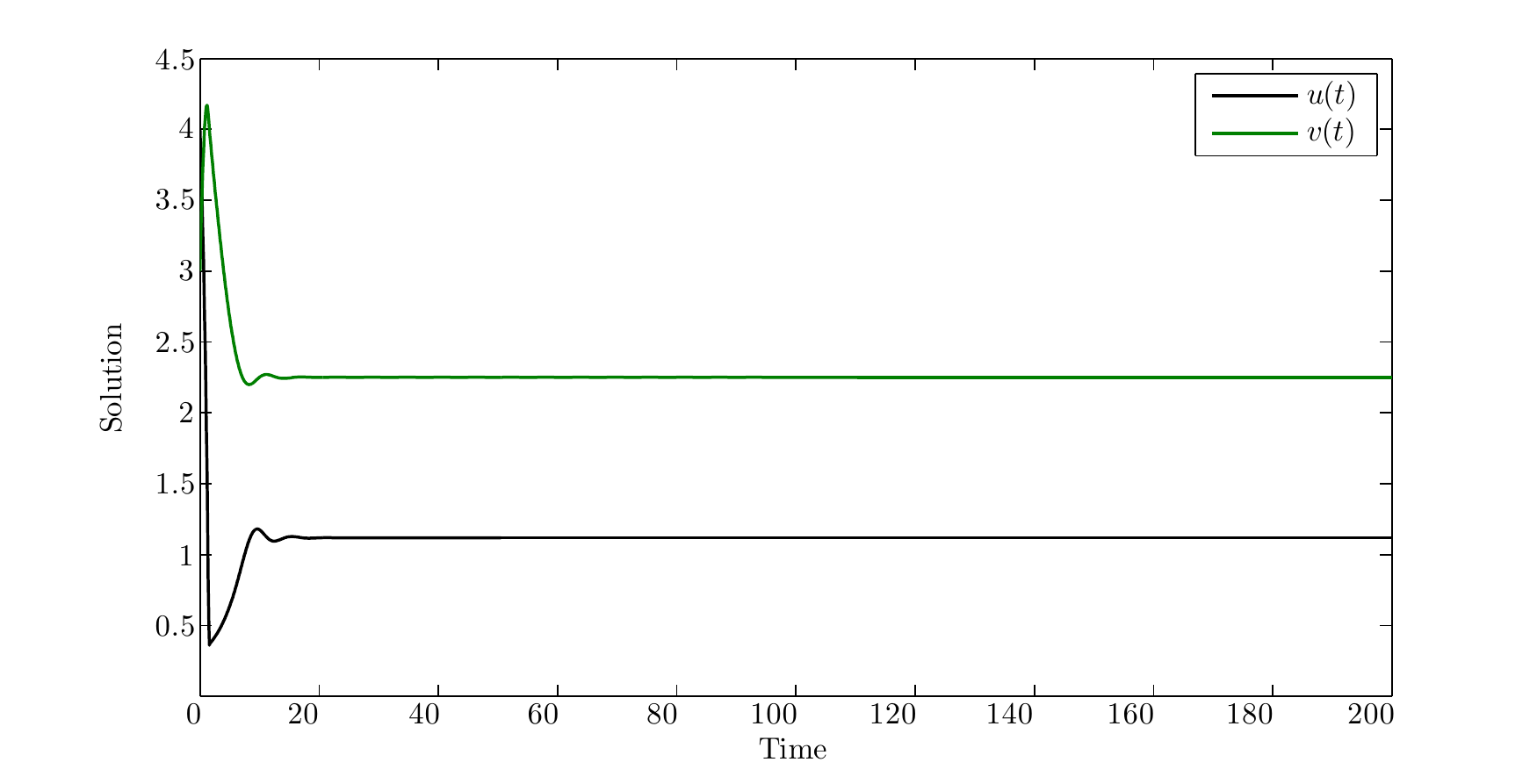}
\caption{Solutions of the Lengyel--Epstein reaction--diffusion model (%
\protect\ref{Ex1}) in the ODE case with the chosen parameters.}
\label{Lengyel_Pap6_1}
\end{figure}

\begin{figure}[tbp]
\centering \includegraphics[width = 4.5in]{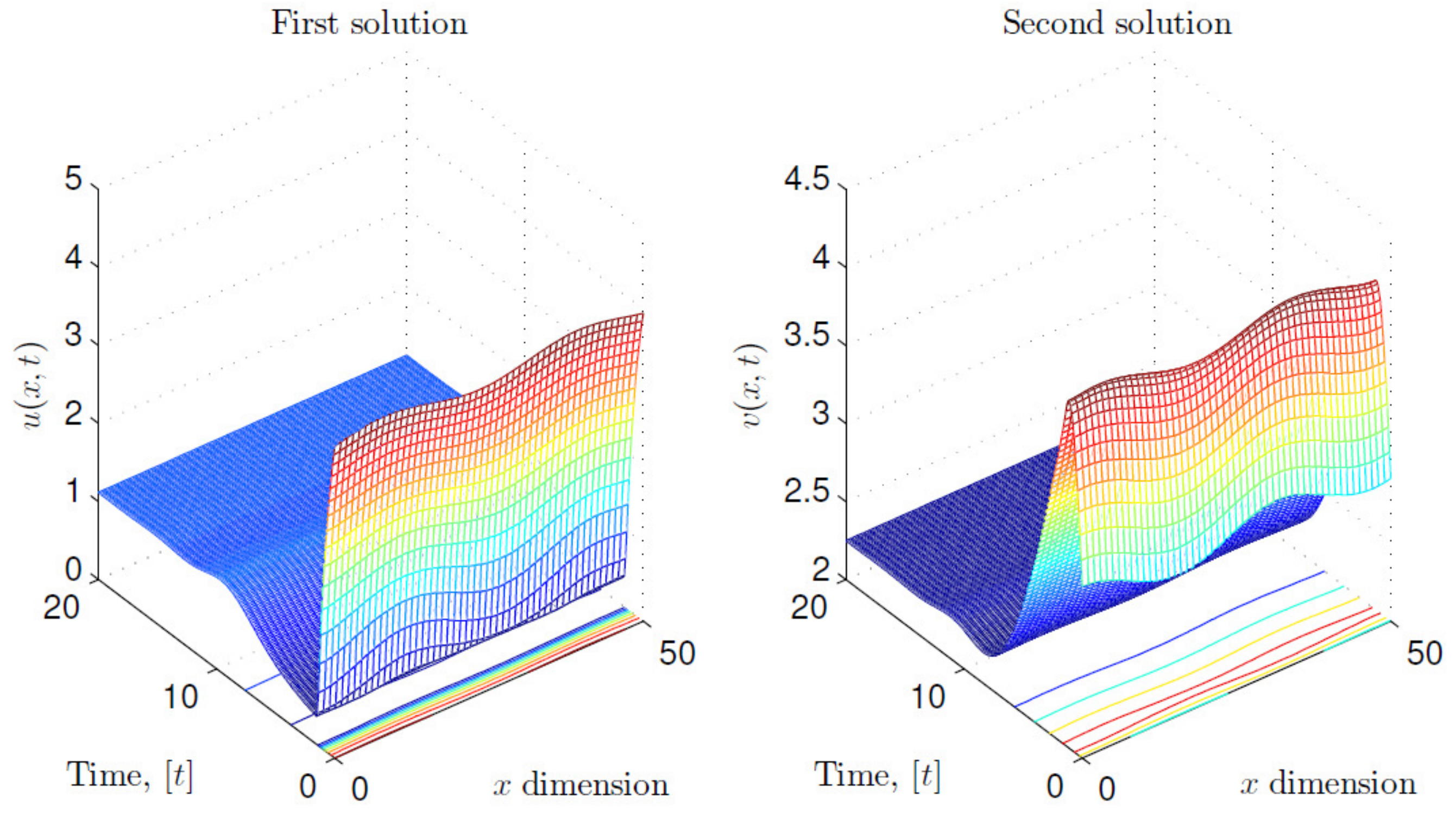}
\caption{Solutions of the Lenyel--Epstein reaction--diffusion model (\protect
\ref{Ex1}) in the one-dimensional diffusion case with the chosen parameters.}
\label{Lengyel_Pap6_2}
\end{figure}

\subsection{FitzHugh--Nagumo Model}

Here, we consider the FitzHugh-Nagumo reaction--diffusion model \cite%
{Oshita2003,Doelman2009}, which represents an excitation system such as a
neuron in Human physiology. The system is of the form%
\begin{equation}
\left \{ 
\begin{array}{l}
\frac{\partial u}{\partial t}=d_{1}\Delta u-u^{3}+\left( 1+\beta \right)
u^{2}-\beta u-v+I \\ 
\frac{\partial v}{\partial t}=d_{2}\Delta v+\varepsilon u-\varepsilon \gamma
v,%
\end{array}%
\right.  \label{FitzModel}
\end{equation}%
where $u\left( x,t\right) $ and $v\left( x,t\right) $ represent the
potential and sodium gating variable in the cell membrane, respectively. The
constants $\beta $, $\gamma $, and $\varepsilon $ are assumed to be positive
with $0<\beta <1$, $\varepsilon \ll 1$ (accounting for the slow kinetics of
the sodium channel). The constant $I$ represents external stimuli. It can be
easily observed that this model falls under the general framework proposed
in this paper, i.e. system (\ref{1.1}), with%
\begin{equation*}
\left \{ 
\begin{array}{l}
f\left( u\right) =-u^{3}+\left( 1+\beta \right) u^{2}-\beta u+I, \\ 
g\left( u\right) =\frac{u}{\gamma }, \\ 
\varphi \left( u\right) =1, \\ 
\lambda =1, \\ 
\sigma =\varepsilon \gamma .%
\end{array}%
\right.
\end{equation*}%
The constant $\delta $ is the solution of $f(u)=0$ as stated in (\ref{con1}%
), i.e.%
\begin{equation*}
-u^{3}+\left( 1+\beta \right) u^{2}-\beta u+I=0.
\end{equation*}

Note here that $\varphi \left( 0\right) \neq 0$ meaning that condition (\ref%
{con1}) is not satisfied. However, as stated in Remark \ref{Phi0Remark},
this does not affect the applicability of our results to this case as the
inequality (\ref{PhiGen}) holds for this particular example. The function $%
g\left( u\right) $ is clearly increasing and satisfies condition (\ref{con2}%
). We consider, for instance, the case where $\beta =0.139$, $\varepsilon
=0.008$, $\gamma =2.54$, and $I=2$. Hence, $\delta =1.7282$. Since, $\lambda
=1$ and the functions $f(u)$ and $g(u)$ intersect at a single point as shown
in Figure \ref{FitzHugh_Ex1_fg}, it is safe to say that the FitzHugh--Nagumo
system also satisfies condition (\ref{con3}). Note that the dynamic range
for $u\ $is $\left( 0,1.7282\right) $ and the unique equilibrium solution of
the system as defined by (\ref{2.2a}) is $\left( u^{\ast },v^{\ast }\right)
=\left( \alpha ,g\left( \alpha \right) \right) =\left( 1.5928,0.6273\right) $%
. Hence, clearly, condition (\ref{con4}) is satisfied.

Let us now examine the local asymptotic stability of the equilibrium in the
ODE scenario. Differentiating $f\left( u\right) $\ and substituting for $%
\alpha =1.5928$ can easily show that stability conditon (\ref{2.3}) is
satisfied. One thing that remains to be examined is the global asymptotic
stability of the system when diffusion is present. It is clear from the
shape of $f(u)$ in Figure \ref{FitzHugh_Ex1_fg} that it fulfills condition (%
\ref{con6}) over the range $u\in \left( 0,\delta \right) $.

The solutions of this particular example were obtained using Matlab
simulations and are depicted in Figures \ref{FitzHugh_Ex1_1} and \ref%
{FitzHugh_Ex1_2} for the ODE and one-dimensional diffusion cases,
respectively. Note that in the ODE case, the initial data is assumed to be%
\begin{equation*}
u\left( 0\right) =0.5\text{ and }v\left( 0\right) =1.2,
\end{equation*}%
whereas in the one-dimensional diffusion case, a slight sinusoidal
perturbation is added%
\begin{equation*}
u\left( x,0\right) =0.5+0.2\sin \left( \frac{x}{5}\right) \text{ and }%
v\left( x,0\right) =1.2+0.2\cos \left( \frac{x}{5}\right) .
\end{equation*}%
Clearly, both in the ODE and one-dimensional diffusion cases, the solutions
are stable and tend to the equilibrium solution.

\begin{figure}[tbp]
\centering \includegraphics[width = 4.5in]{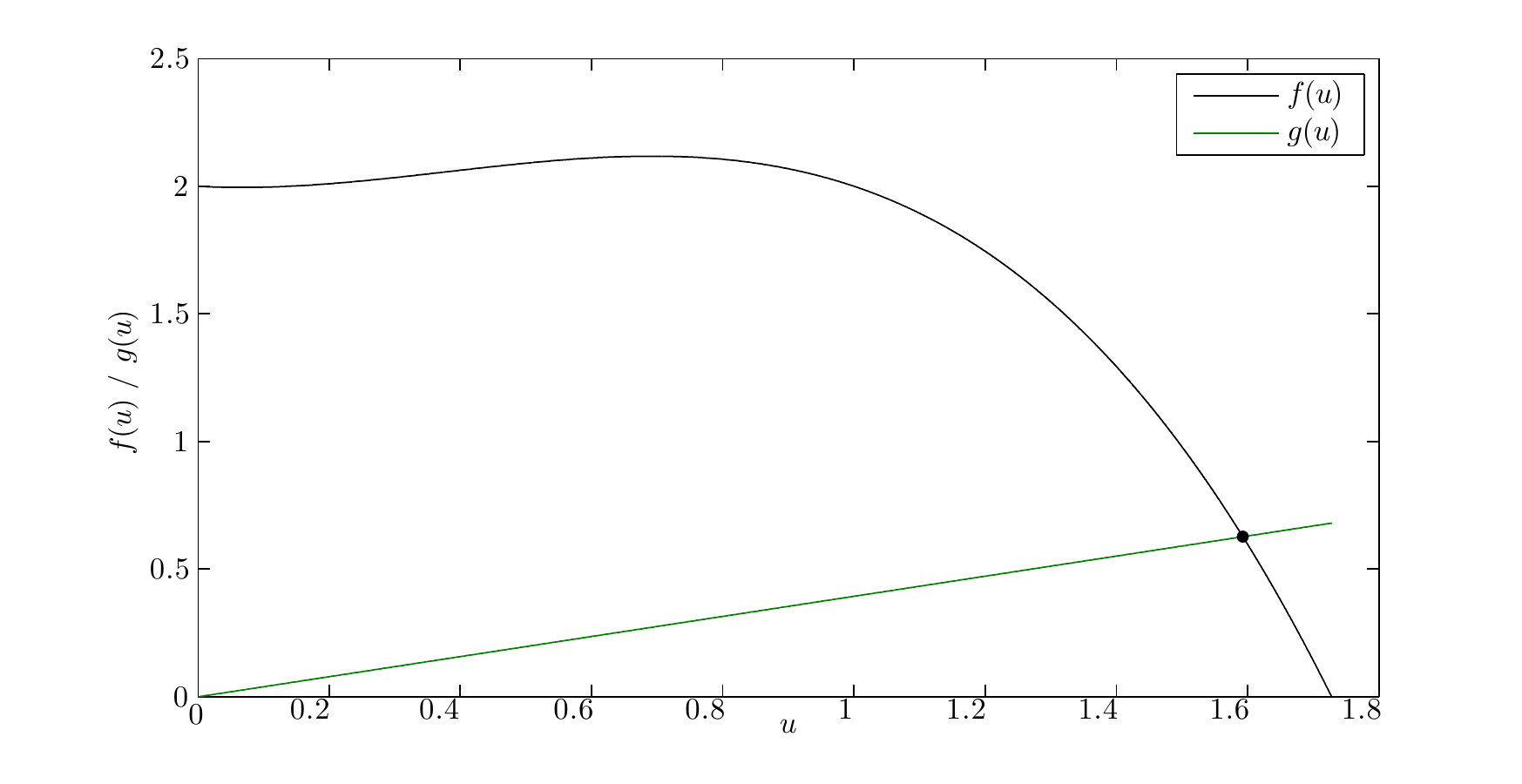}
\caption{The shape and intersection of functions $f(u)$ and $g(u)$ for the
FitzHugh-Nagumo reaction--diffusion model.}
\label{FitzHugh_Ex1_fg}
\end{figure}

\begin{figure}[tbp]
\centering \includegraphics[width = 4.5in]{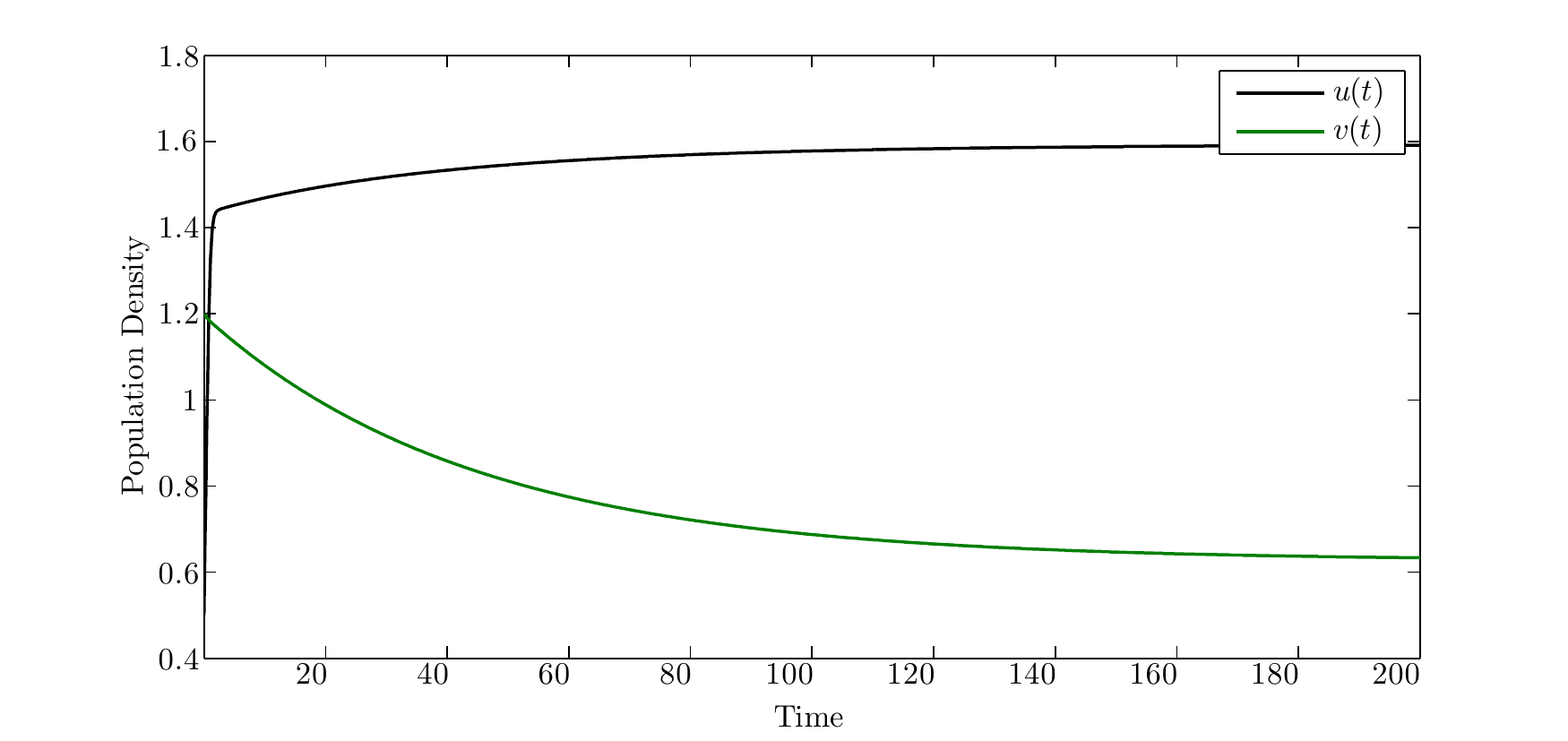}
\caption{Solutions of the FitzHugh-Nagumo reaction--diffusion model (\protect
\ref{FitzModel}) in the ODE case with the chosen parameters.}
\label{FitzHugh_Ex1_1}
\end{figure}

\begin{figure}[tbp]
\centering \includegraphics[width = 4.5in]{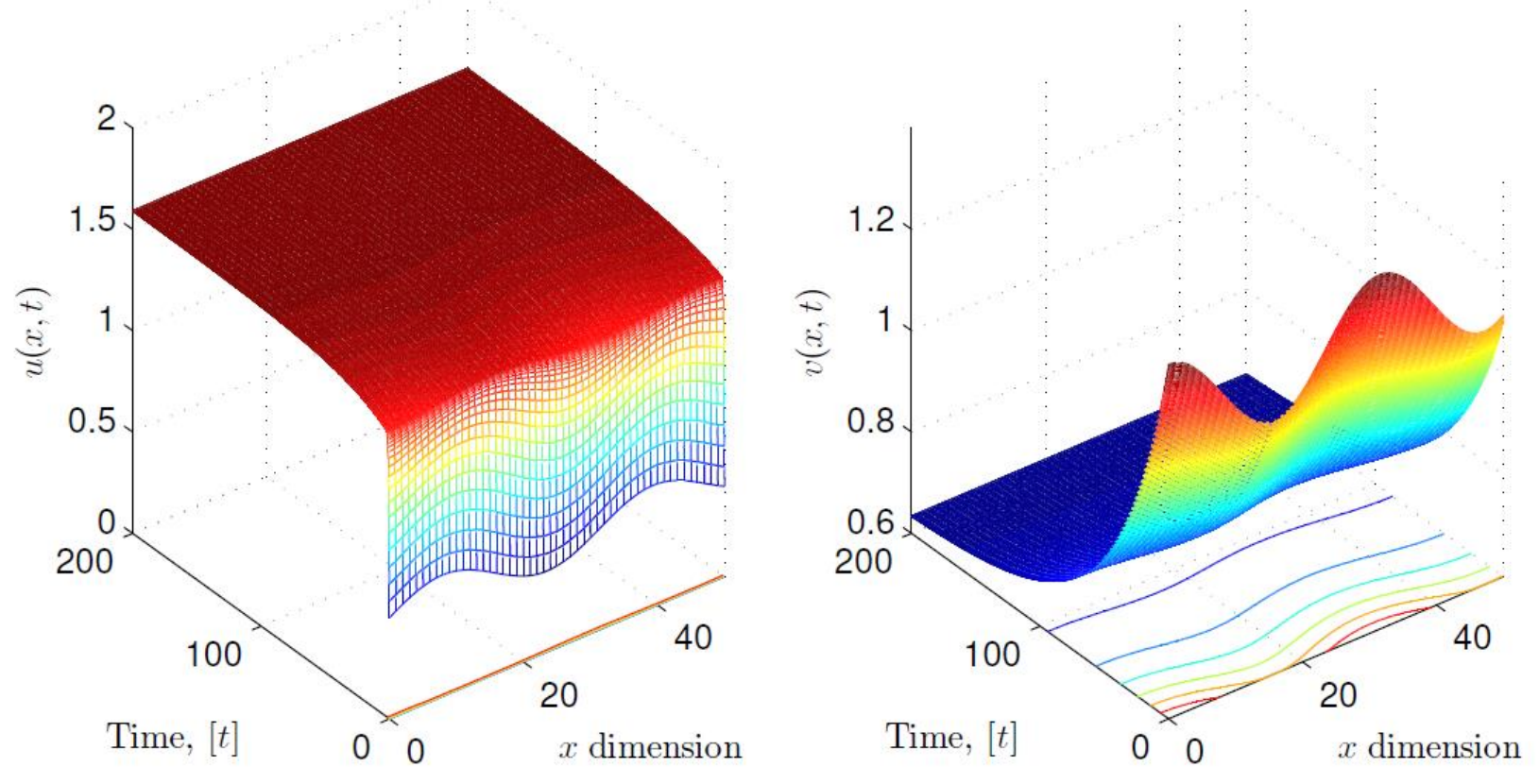}
\caption{Solutions of the FitzHugh-Nagumo reaction--diffusion model (\protect
\ref{FitzModel}) in the one-dimensional diffusion case with the chosen
parameters.}
\label{FitzHugh_Ex1_2}
\end{figure}

\FloatBarrier

\end{document}